\documentclass[11pt,a4paper]{article}
\usepackage{amssymb,amsmath,amsthm,bm}
\usepackage{graphicx}
\usepackage{geometry}
\usepackage{float}
\usepackage{tikz}
\graphicspath{{./figures/}}
\usepackage[maxbibnames=99]{biblatex}
\newtheorem{teo}            {Theorem}
\newtheorem{lema}     [teo]{Lemma}
\newtheorem{defin}[teo]{Definition}
\newtheorem{corollary} [teo]{Corollary}

\newtheorem{example} [teo]          {Example}
\newtheorem{obs} [teo]           {Remark}
\newtheorem{prop} [teo]       {Proposition}

\newcommand{\mc}{\multicolumn{1}{c}}

\DeclareMathOperator{\ric}{Ric}
\DeclareMathOperator{\Ric}{Ric}

\title{Ricci curvature, graphs and eigenvalues}
\author{Viola Siconolfi\ \footnotemark[1]}
\begin{document}
\maketitle
\footnotetext[1]{
Fakult\"at f\"ur Mathematik, Universit\"at Bielefeld, Germany.\\
\textit{Email address: }\texttt{vsiconolf@math.uni-bielefeld.de}.} 
\footnotetext[2]{
An extended abstract containing some of the results of this paper appeared in the proceedings of FPSAC 2020 \cite{sic}}

\begin{abstract}
We express the discrete Ricci curvature of a graph as the minimal eigenvalue of a family of matrices, one for each vertex of a graph whose entries depend on the local adjaciency structure of the graph. Using this method we compute or bound the Ricci curvature of Cayley graphs of finite Coxeter groups and affine Weyl groups. As an application we obtain an isoperimetric inequality that holds for all Cayley graphs of finite Coxeter groups.
\end{abstract}
Key Words: Ricci curvature, Cayley graphs, Coxeter Theory, graph Theory\\

MSC2020: 20F55, 53A70, 05C50 

\section{Introduction}
There is a trend in graph theory of studying discrete versions of concepts that classically appear in differential geometry (see for example \cite{bre7},\cite{bre6},\cite{bre5},\cite{bre17},\cite{bre12},\cite{bre13},\cite{bre16},\cite{bre9},\cite{bre4},\cite{bre15}\\,\cite{bre3},\cite{bre11},\cite{bre10},\cite{bre8},\cite{bre14},\cite{bre1}). In particular, a definition of Ricci curvature for graphs was given in 1999 by Schmuckenschl{\"a}ger. The intuition behind this definition is to translate in a discrete setting a Bochner-type inequality (see \cite{shmuck}). Bounds on this Ricci curvature imply a bound on the spectral gap, a number of isoperimetric inequalities (see \cite{isop}) and a bound on the norm of some heat kernel operators (see \cite[Sections 3 and 4]{klart}). Despite these applications, there are only a few families of graphs for which bounds, or exact values, of this Ricci curvature are known (see \cite{klart}).

In this paper we describe some general results about the computation of the Ricci curvature and then apply them to obtain exact values, or bounds, on the Ricci curvature of certain families of graphs that play a role in Coxeter theory.

Our main general result is Theorem \ref{teomatrix} which reduces the computation of the Ricci curvature of a graph to that of the minimum eigenvalue of a family of matrices, one for each vertex of the graph. 
This result enables us to use Gershgorin's Theorem to study the values of the Ricci curvature of particular families of graphs. In particular we bound the Ricci curvature of trees and in general of graphs with no triangles nor quadrilaterals.

We then focus on Cayley graphs of Coxeter groups with the simple reflections as set of generators, equivalently, the Hasse graphs of the weak order, see e.g.\cite[Chapter 3]{BB}. With a reasoning similar to the one carried out for Theorem \ref{teomatrix} we find the Ricci curvature of these graphs for all finite Coxeter groups and some affine Weyl group. The study of the Ricci curvature through eigenvalues in this case leads back to known formulas for the eigenvalues of tridiagonal and circulant Toeplitz matrices.
As a consequence we obtain an isoperimetric inequality that holds for all finite Coxeter groups.

The motivation to study these graphs comes also from the paper by Klartag, Kozma, Tetali and Ralli in \cite{klart} where some results about the Ricci curvature of Cayley graphs are obtained.

This paper is divided in three sections.
The next section contains preliminaries. We recall the definition of Ricci curvature and some basic results about it. We then state some linear algebra results that we need in the sequel (in sections 3 and 4), and conclude with an introduction about Coxeter groups and Bruhat order.

Section 3 is devoted to new results about the Ricci curvature of general graphs. Our main result gives a constructive method to compute the local Ricci curvature of a graph (Theorem \ref{teomatrix}).
As corollaries we obtain bounds on the Ricci curvature of particular families of graphs, and we present a graph whose Ricci curvature is $-\infty$.

The fourth section deals with Cayley graphs associated to finite Coxeter groups and to affine irreducible Weyl groups. In most cases we obtain the exact value of the curvature of these graphs, in others we give a bound. As a corollary of these results we describe an isoperimetric inequality that holds for Cayley graphs of all finite Coxeter groups.

\section{Preliminaries}\label{prelimin}

\subsection{Ricci curvature of a locally finite graph}\label{radici}
We recall here the definitions and main facts about the discrete Ricci curvature of a locally finite graph, the main reference for this section is Subsection $1.1$ from \cite{klart}, the notation is slightly different.  

Let $G$ be a graph, 
we assume it to be undirected, with no loops and with no multiple edges (i.e. a simple graph); also we ask that it has no isolated vertices.
We denote by $\mathcal{V}(G)$ the set of vertices of $G$, by $\mathcal{E}(G)$ its edges and we take $\delta(x,y)$ as the function $\delta:\mathcal{V}(G)\times\mathcal{V}(G)\rightarrow \mathbb{N}\cup \{\infty\}$ that gives the distance between two vertices. For a fixed $x\in\mathcal{V}(G)$ and $i\in\mathbb{N}$ we define the set:
\[
B(i,x):= \{u\in\mathcal{V}(G)|\delta(x,u)=i\}; 
\]
we denote by $d(x)$ the cardinality of $B(1,x)$ and call it the degree of $x$. 

\begin{defin}
We say that a given $G$ is locally finite if $d(x)< \infty$ for all $x\in \mathcal{V}(G)$.
\end{defin}
From now on we assume $G$ to be locally finite.

Given real functions $f$ and $g$ on $\mathcal{V}(G)$ and $x\in\mathcal{V}(G)$ we define the following operators:
\begin{itemize}
  \item $\Delta (f)(x):=\sum_{v\in B(1,x)}(f(v)-f(x))$;
  \item $\Gamma (f,g)(x):=\frac{1}{2}\sum_{v\in B(1,x)}(f(x)-f(v))(g(x)-g(v))$;
  \item $\Gamma_2(f)(x):=\frac{1}{2}\left(\Delta(\Gamma(f,f))(x)\right)-\Gamma(f,\Delta(f))(x)$.
\end{itemize}
Instead of $\Gamma(f,f)(x)$ we write $\Gamma(f)(x)$. Note that $\Gamma(f)(x)\geq 0$ $\forall x\in \mathcal{V}(G)$, the equality holds if and only if $f(x)=f(v)$ for all $v$ in $B(1,x)$. Also note that $\Delta(f)(x)=\sum_{v\in B(1,x)}f(v)-d(x)f(x)$.

\medskip
These definitions allow us, following \cite{klart} and \cite{shmuck}, to introduce the Ricci curvature of a graph:
\begin{defin}\label{curvy}
The discrete Ricci curvature of a graph $G$, denoted $\ric(G)$, is the maximum value $K\in\mathbb{R}\cup \{-\infty\}$ such that for any real function $f$ on $\mathcal{V}(G)$ and any vertex $x$, $\Gamma_2(f)(x)\geq K\Gamma(f)(x)$ .
\end{defin}
There is also a local version of this definition: 
\begin{defin}
The local Ricci curvature of a graph $G$ at a given point $\bar{x}\in\mathcal{V}(G)$ is the maximum value $K\in\mathbb{R}\cup \{-\infty\}$ such that for any real function $f$ on $\mathcal{V}(G)$ , $\Gamma_2(f)(\bar{x})\geq K\Gamma(f)(\bar{x})$ holds. The local curvature so defined is denoted $\Ric(G)_{\bar{x}}$. 
\end{defin}

We obtain that $\Ric(G)=inf_{x\in\mathcal{V}(G)}\Ric(G)_x$.

For brevity, in the rest of this work, we often say "curvature" instead of "Discrete Ricci curvature".

\begin{obs}
Note that if $c,d\in\mathbb{R}$
and $f:\mathcal{V}(G)\rightarrow \mathbb{R}$
then $\Delta(f)=\Delta(f+c)$, $\Gamma(f+c,g+d)=\Gamma(f,g)$ and $\Gamma_{2}(f+c)=\Gamma_2(f)$. 
We may therefore assume that in the expression $\Gamma_2(f)(x)\geq K\Gamma(f)(x)$, $f$ satisfies $f(x)=0$. This allow us to use the following formula for $\Gamma$:
$$\Gamma (f)(x)=\frac{1}{2}\sum_{v\in B(1,x)}f(v)^2.$$
\end{obs}

The following result appears in \cite[Subsection 1.1]{klart}:

\begin{prop}\label{proof2}
$\Gamma_2$ can be expressed through the following formula when $f(x)=0$:
$$2\Gamma_2(f)(x)=\frac{1}{2}\sum_{u\in B(2,x)}\sum_{v\in B(1,u)\cap B(1,x)}(f(u)-2f(v))^2+\left( \sum_{v\in B(1,x)}f(v) \right)^2+
$$
$$
+\sum_{\{v,v'\}\in \mathcal{E}(G)}\left(2(f(v)-f(v'))^2+\frac{1}{2}(f(v)^2+f(v')^2)\right)+\sum_{v\in B(1,x)}\frac{4-d(x)-d(v)}{2} f(v)^2,
$$
where the third sum runs over all $v,v'\in B(1,x)$ such that $\{v,v'\}$ is an edge in $G$.
\end{prop}

The following simple observation does not appear anywhere in the literature so we include its proof:

\begin{lema}\label{proof3}
We can write
  \begin{equation}\label{equcurv}
  \text{Ric}(G)=\text{inf}_{x,f} \frac{\Gamma_2(f)(x)}{\Gamma(f)(x)} 
  \end{equation}
  where $x$ ranges in $\mathcal{V}(G)$ and $f$ ranges over the real functions defined on $\mathcal{V}(G)$ such that $f(x)=0$ and $\Gamma(f)(x)>0$.
\end{lema}

\begin{proof}
We know that $G$ has no isolated vertices, so there is a function $\bar{f}$ on $\mathcal{V}(G)$ such that $\Gamma(\bar{f})(x)> 0$. Any $K$ that satisfies 
$$\Gamma_2(f)(x)\geq K\Gamma(f)(x )\quad \forall x,\forall f$$  also satisfies
$$
K\leq \frac{\Gamma_2(g)(x)}{\Gamma(g)(x)}<\infty
\quad \forall x,g \text{ s.t.} \Gamma(g)(x)>0.
$$
Let now $\tilde{f}:\mathcal{V}(G)\rightarrow \mathbb{R}$ be such that $\Gamma(\tilde{f})(x)=0$, it comes easily from Proposition \ref{proof2} that in this case $\Gamma_2(\tilde{f})(x)\geq 0$ and so in particular $\Gamma_2(\tilde{f})(x)\geq K\Gamma(\tilde{f})(x)$.
The statement follows. 
\end{proof}

\begin{obs}
From Lemma \ref{proof3} we deduce the following formula for the local Ricci curvature:
\[
  \text{Ric}(G)_x:=\text{inf}_{f} \frac{\Gamma_2(f)(x)}{\Gamma(f)(x)},
\]
where $f$ ranges among the functions on $\mathcal{V}(G)$ such that $f(x)=0$ and $\Gamma(f)(x)\neq 0$.
\end{obs}

\begin{obs}\label{obs2}
From the formulas for $\Gamma$ and $\Gamma_2$ we obtain that $\Ric(G)_x$ depends only on the subgraph obtained as the union of the paths of length 1 and 2 starting from $x$. The set of vertices of this graph is $\{x\}\cup B(1,x)\cup B(2,x)$ and the edges are the ones connecting vertices in $B(1,x)$ to vertices in $\{x\}\cup B(1,x)\cup B(2,x)$. We will refer to such a subgraph as the length-2 path subgraph of $x$. As a consequence two vertices with isomorphic length-2 path subgraphs have the same local Ricci curvature.
\end{obs}

We conclude with the following remark about triangle-free graphs. These are graphs with no $x,u,v\in\mathcal{V}(G)$ such that $\{x,u\},\{x,v\},\{u,v\}\in \mathcal{E}(G)$:

\begin{obs}\label{triang}
If $G$ is a triangle-free graph, then:
\[
2\Gamma_2(f)(x)=\frac{1}{2}\sum_{u\in B(2,x)}\sum_{v\in B(1,u)\cap B(1,x)}(f(u)-2f(v))^2+
\]
\[
+(\sum_{v\in B(1,x)}f(v))^2+\sum_{v\in B(1,x)}\frac{4-d(x)-d(v)}{2} f(v)^2.
\]
We obtain this formula just by erasing the third sum in the equation in Proposition \ref{proof2}.
\end{obs}

We go on by recalling \cite[Theorem 1.2]{klart} and a corollary:

\begin{teo}\label{teotriang}
Let $G$ be a locally finite graph, $t(v,v')$ be the function that counts the number of triangles containing both the vertices $v$ and $v'$ and let $T=sup_{\{v,v'\}\subset \mathcal{V}(G)} t(v,v')$. Then $\Ric(G)\leq 2+\frac{T}{2}$.
\end{teo}

\begin{corollary}\label{coro}
Let $G$ be a graph with no triangles, then $\Ric(G)\leq2$.
\end{corollary}

%\subsection{Applications of Ricci curvature}\label{ssappli}
Knowing the Ricci curvature of a graph allows one to obtain other information about it. We conclude by presenting some results about the spectral gap and the isoperimetry of a graph that depend on a lower bound on the Ricci curvature, and that are used in this work, our references are \cite[Section 3]{klart} and \cite[Section 4]{klart}. To learn more about isoperimetric inequalities for graphs we refer the reader to \cite[Section 4]{isop}.  From now until the end of the section we consider only finite graphs.

%\subsubsection{Spectral gap and isoperimetry}

 We start by defining the spectral gap, for a more detailed overview of spectral graph theory the reader can refer to \cite{chung}.

\begin{defin}
The adjacency matrix of $G$ is a square matrix of dimension $|\mathcal{V}(G)|\times|\mathcal{V}(G)|$ such that the entry $A_{ij}$ of the matrix is $1$ if there is an edge between vertex $i$ and vertex $j$ of $G$, $A_{ij}$ is $0$ otherwise.

The Laplacian matrix is defined as:
\[
\overline{\Delta}:=-D+A
\]
where $D$ is the diagonal matrix whose entries are the degrees of the vertices and $A$ is the adjacency matrix of the graph.
\end{defin}

\begin{defin}\label{spec}
Let $G$ be a graph and $\overline{\Delta}$ its Laplacian matrix. The spectral gap of $G$, $\lambda_G$ is the least non zero eigenvalue of $-\overline{\Delta}$.
\end{defin}
 Where no ambiguity occurs we shall just denote the spectral gap as $\lambda$.
The following results, proved in \cite{klart}, show how the Ricci curvature of a graph can help bound its spectral gap and how these two values allow an isoperimetric inequality.

\begin{teo}\label{spgap}
Let $G$ be a graph with $\Ric(G)>0$, then $\lambda\geq \Ric(G)$.
\end{teo}

\begin{teo}\label{sgiso}
Suppose $G$ has $\Ric(G)\geq K$, for some $K\in \mathbb{R}\setminus \{0\}$. Then, for any subset $A\subset \mathcal{V}(G)$,
\[
|\partial A| \geq \frac{1}{2}min \{\sqrt{\lambda},\frac{\lambda}{\sqrt{2|K|}}\}|A|\left(1-\frac{|A|}{|\mathcal{V}(G)|}\right) .
\]
Here, by $\partial A$, we mean the collection of all edges connecting $A$ to its complement.
\end{teo}

Notice that for Theorem \ref{spgap} it is crucial to have $\Ric(G)$  positive,  but positivity is not required for Theorem \ref{sgiso}. With a different bound on the spectral gap it is still possible to apply Theorem \ref{sgiso} to families of graphs with a negative lower bound for the Ricci curvature.
Finally, we recall the following result that bounds the spectral gap of any Cayley graph:

\begin{corollary}\label{caysg}
Let $G$ be a finite group and $S\subset G$ be a set of generators. We denote by $Cay(S)$ the Cayley graph of $G$ defined by taking $S$ as set of generators. Then:
\[
\lambda_{Cay(S)}\geq \frac{|G|}{d|S|^d}.
\]
Where $d$ is the diameter of $Cay(S)$.
\end{corollary}
This previous result is proved in \cite{spgap}.

\subsection{Matrices and eigenvalues}

In Section \ref{subsec} we need to study the maximum eigenvalue of some matrices. For this purpose the results that are presented in this subsection are of great usefullness. First we recall a result that bounds the eigenvalues of any matrix:

\begin{teo}[Gershgorin]\label{gersh}
Let $M=(m_{i,j})$ be a matrix in $\mathbb{R}^{n\times n}$, let $R_i=\sum_{j=1,j\neq i}^{n}|m_{i,j}|$ and
\[
K_i:=\{z\in \mathbb{C}||z-m_{i,i}|\leq R_i\}.
\]
Then all the eigenvalues of $M$ are contained in $\cup_{i=1}^{n}K_i$. The $K_i$ are called Gershgorin circles. 
\end{teo}

A proof of Gershgorin's Theorem can be found, e.g., in \cite[Subsection 7.2.1]{gol} (see also \cite{gers}).

We now recall some formulas for the eigenvalues of some particular families of matrices.

Let us consider the following matrix:
\[
M(\alpha,\beta,a,b,c)=
\begin{bmatrix}
-\alpha+b&c&&&\\
a&b&\ddots&\textrm{\Large{0}}&\\
&\ddots&\ddots&\ddots&\\
&\textrm{\Large{0}}&\ddots& b &c \\
&&&a&-\beta+b\\
\end{bmatrix}
\quad \forall (\alpha,\beta,a,b,c)\in \mathbb{R}^5.
\]
Its eigenvalues can be computed using the next result first proved in \cite[Theorem 4]{wen}:
\begin{teo}
\label{teowen}
Suppose that $ac>0$ and that $\alpha=\beta=\sqrt{ac}\neq 0$, then the eigenvalues $\lambda_1,\ldots,\lambda_n$ of $M(\alpha,\beta,a,b,c)$ are given by
\[
\lambda_i=b+2\sqrt{\alpha\beta}\;\cos(\frac{i\pi}{n}), \quad i=1,\ldots,n.
\]
\end{teo}

Finally we consider one last family of matrices:
\[
M(c_0,\ldots,c_{n-1})=
\begin{bmatrix}
c_0    &c_{n-1}&\ldots & c_2  & c_1   \\
c_1    &c_0    &c_{n-1}&      & c_2   \\
\vdots &c_1    &   c_0 &\ddots&\vdots \\
c_{n-2}&       &\ddots &\ddots&c_{n-1}\\
c_{n-1}&c_{n-2}&\ldots &c_1   &c_0    \\
\end{bmatrix}
\quad (c_0,\ldots,c_{n-1})\in \mathbb{R}^{n}.
\]
These matrices are such that each descending diagonal from left to right is constant (matrices with this property are called Toeplitz matrices, see e.g. \cite{gray}). Furthermore the values on the diagonals are repeated as pictured, in this case we call these circulant Toeplitz matrices. The following theorem (\cite[Section 3.1]{gray}) holds:
\begin{teo}\label{circ}
The eigenvalues $\lambda_{1},\ldots,\lambda_{n}$ of $M(c_0,\ldots,c_{n-1})$ are given by
\[
\lambda _{j}=c_{0}+c_{n-1}\zeta^{j}+c_{n-2}\zeta^{2j}+\ldots +c_{1}\zeta^{(n-1)j},\qquad j=0,1,\ldots ,n-1. 
\]
Where $\zeta=e^{\frac{2i\pi }{n}}$.
\end{teo}

\subsection{Coxeter groups}\label{ssec1}

In this subsection we recall the definition and main facts about Coxeter groups, the main reference is \cite{BB}. Our main goal is to define a family of graphs associated to Coxeter groups, namely the weak order graphs ( denoted $V(W)$).

 Coxeter systems are pairs $(W,S)$ where $W$ is a group generated by the elements in $S$	 and $S=\{s_i\}_{i\in I}$ is a finite set with the following relations:

\[
(s_is_j)^{m_{i,j}}=e.
\] 
The values $m_{i,j}$ are usually seen as the entries of a symmetric matrix with $m_{i,i}=1$ for all $i\in I$ and $m_{i,j}\geq 2$ (including $m_{i,j}=\infty$) for all $i,j\in I, i\neq j$. $W$ is called a Coxeter group, $S$ turns out to be a minimal set of generators, its elements are called Coxeter generators.
All the information about a Coxeter group can be encoded in a labeled graph called the Coxeter graph. Its set of vertices is $S$, and there is an edge between two vertices $s_i$ and $s_j$ if $m_{i,j}\geq 3$, such an edge is labeled with $m_{i,j}$ if $m_{i,j}\geq 4$. 
We state a very classical result from Coxeter Theory, namely the classification of finite Coxeter groups, for more details the reader can see \cite[Appendix A1]{BB} and \cite[Chapter 2]{hump}.
\begin{teo}
Given a finite Coxeter group $W$, this can be written in a unique way as direct product of the following irreducible Coxeter groups:
\[
W=W_1\times\ldots\times W_k
\] 
With $W_i$ of the following kind:
\begin{itemize}
\item $A_n$ $n\geq 1$;
\item $B_n$ $n\geq 2$;
\item $D_n$ $n\geq 3$;
\item $I_2(m)$ $m\geq 2$;
\item $H_3$, $H_4$;
\item $E_6$, $E_7$, $E_8$;
\item $F_4$.
\end{itemize}
\end{teo}
\begin{obs}
Groups of type $A_n$ are symmetric groups, in particular $A_{n}=S_{n+1}$. The groups of type $B_n$, called hyperoctahedral groups, are the groups of permutations $\pi$ of the set $\{\pm 1,\ldots,\pm n\}$ such that $\pi(a)=-\pi(a)$. Finally $D_n$ is the subgroup of $B_n$ of permutations such that an even number of positive elements has negative image. $D_n$ is called even hyperoctahedral group.
\end{obs}

We define strictly linear Coxeter systems following \cite{marietti}:
\begin{defin}
Given a Coxeter system $(W,\{s_1,\ldots,s_n\})$, this is called strictly linear if:
\begin{itemize}
    \item$m_{i,j} \geq 3$ if $|i-j|=1$;
    \item$m_{i,j}=2$ if $1<|i-j|$;
\end{itemize}
\end{defin}
Notice that these groups are the ones whose Coxeter graph is a path.

We continue with some classical definitions in Coxeter theory. 
Given an element $w$ in $W$, this can be written as a product of elements in $S$
\[
w=s_1\ldots s_k. 
\]
If $k$ is the minimal length of all the possible expressions for $w$, we say that $k$ is the length of $w$ and we write $\ell(w)=k$. We define in $W$ the set of reflections as the union of all the conjugates of $S$, $T:=\cup_{w\in W}wSw^{-1}$. The definitions of length and reflections allow us to define two partial orders on the set $W$:

\begin{defin}
Given $w\in W$ and $t\in T$, if $w'=tw$ and
$\ell(w')<\ell(w)$ we write $w'\rightarrow w$.
 Given two elements $v,w\in W$ we say that 
 $v \geq w$ according to the Bruhat order if there are 
 $w_0,\ldots, w_k\in W$ such that
\[
v=w_0\leftarrow w_1\ldots w_{k-1}\leftarrow w_k=w.
\]
\end{defin}
Thus we obtain a first partial order on $W$. The undirected Bruhat graph associated to a Coxeter group, denoted $B(W)$, is the graph whose set of vertices is $W$ and such that there is an edge between two vertices $w,v$ if and only if $w\rightarrow v$ or $v\rightarrow w$.

The second order we define on the elements of a Coxeter group is the weak order on $(W,S)$:
\begin{defin}
 Given $u,v\in W$, we say that $u\leq_{R} w$ if $w=us_1,\ldots s_k$, for some $s_i\in S$ such that $\ell(us_1\ldots s_i)=\ell(u)+i$ for every $1\leq i\leq k$.
\end{defin}
This is the definition of the right weak order on $W$, analogously one defines the left weak order, multiplying the $s_1,\ldots s_k$ on the left. These two orders are isomorphic via the map $w\rightarrow w^{-1}$.

Given a Coxeter group $(W,S)$ we denote by $V(W)$ the Hasse graph associated to the weak order, we call it simply 'weak order graph'. This has set of vertices $W$, two elements in the graph are connected by an (undirected) edge if one covers the other according to the weak order.

\section{Eigenvalues and discrete Ricci curvature}\label{subsec}

What follows can be seen as a sequel of Subsection \ref{radici}, after introducing Ricci curvature for graphs here we present some new results on its computation.  The main result is \ref{teomatrix} which describes a method to recover the local Ricci curvature of a graph by computing the minimal eigenvalue of a given matrix. Among the consequences of Theorem \ref{teomatrix} we find some inequalities to bound the value of discrete Ricci curvature for graphs which satisfy some conditions. These conditions are, as seen on Theorem \ref{teotriang}, about the number of triangles of the graph or, similarly, on the number of quadrilaterals.

First we notice that graph automorphisms respect the Ricci curvature: 
\begin{lema}\label{stesso}
If $G$ is any locally finite graph and $\chi$ is a graph automorphism then the following holds:
$$
Ric(G)_x=Ric(G)_{\chi(x)}.
$$
\end{lema}
\begin{proof}
Trivial from Remark \ref{obs2}.
\end{proof}

Now we introduce our main result about the computation of discrete Ricci curvature:

We fix $x\in\mathcal{V}(G)$, for any $v\in B(1,x)$ we denote by $\mathcal{U}_v:=B(2,x)\cap B(1,v)$ and by
$t_v$ the number of triangles containing vertices $v$ and $x$, so $t_v=|\{v'\in B(1,x)|\{v,v'\}\in \mathcal{E}(G)\}|$ . For any $u\in B(2,x)$ we denote by $n_u$ the cardinality of $B(1,u)\cap B(1,x)$, this is  the number of paths of length $2$ joining $x$ and $u$. We define a function $T:B(1,x)\times B(1,x)\rightarrow \{0,1\}$ as follows:
\[
T(v,v')= \begin{cases} 1 &\mbox{if } \{v,v'\}\in \mathcal{E}(G) \\ 
0 &\mbox{ otherwise }  \end{cases}.
\]

\begin{defin}
Given a locally finite graph $G$ and a vertex $x$ of $G$. Assume that $d(x)=|B(1,x)|=d$, and let us define an order on the elements adjacent to $x$ that makes $B(1,x)=\{v_1,\ldots,v_d\}$. We define the following matrix associated to $x$:
\[
A_{ij}(x) = \begin{cases} \sum_{u\in\mathcal{U}_{v_i}}\frac{2(n_u-1)}{n_u}+1+\frac{4-d(x)-d(v_i)}{2}+\frac{3}{2}t_{v_i} &\mbox{if } i=j \\ 
\sum_{u\in \mathcal{U}_{v_i}\cap\mathcal{U}_{v_j}}(-\frac{2}{n_u})+1+2T(v_i,v_j) & \mbox{if } i\neq j \end{cases},
\]
for all $i,j\in\{1,\ldots,n\}$.
\end{defin}
An example of the construction of such a matrix can be found in Example \ref{esempio3}.

\begin{teo}\label{teomatrix}
Given a locally finite graph $G$ and $x$ a vertex of $G$, then

\[
\Ric(G)_x=\min \{\lambda|\lambda \text{ is an eigenvalue of }A(x)\}.
\]
As a consequence
\[
\Ric(G)=\inf\{\lambda|\lambda \text{ is an eigenvalue of }A(x),x\in \mathcal{V}(G)\}.
\]

\end{teo}

\begin{proof}

Let $u\in B(2,x)$, we notice that
\begin{align*}
\sum_{v\in B(1,x)\cap B(1,u)}(f(u)-2f(v))^2&=
n_uf(u)^2+\sum_{v\in B(1,x)}4f(v)^2-\\& 4f(u)\sum_{v\in B(1,x)\cap B(1,u)}f(v)\\
&=\left(\sqrt{n_u}f(u)-\frac{2}{\sqrt{n}}\sum_{v\in B(1,x)}f(v)\right)^2+\\
&\frac{4}{n_u}\sum_{\{v,v'\}\subset B(1,x)\cap B(1,u)}(f(v)-f(v'))^2\\
  & \geq
\frac{4}{n_u}\sum_{\{v,v'\}\subset B(1,x)\cap B(1,u)}
(f(v)-f(v'))^2.
\end{align*}
As $f$ varies among the functions on $\mathcal{V}(G)$ such that $f(x)=0$ and $\Gamma(f)(x)\neq 0$. The minimal value of $\sum_{v\in B(1,x)\cap B(1,u)}(f(u)-2f(v))^2$ is achieved for those $f$ satisfying $f(u)=\frac{2}{{n_u}}\sum_{v}f(v)$.

Therefore
\[
Ric_x(G)=\inf_{f}\frac{p(x)+(\sum_{v\in B(1,x)}f(v))^2+\sum_{v\in B(1,x)}\frac{(4-d(x)-d(v))}{2}f(v)^2}{\sum_{v\in B(1,x)}f(v)^2}
\]
\[
+\frac{\sum_{v,v'\in B(1,x),\{v,v'\}\in \mathcal{E}(G)}[2(f(v)-f(v'))^2+\frac{1}{2}(f(v)^2+f(v')^2)]}
{(\sum_{v\in B(1,x)f(v)}f(v))^2}.
\]
Where
\[
p(v)=\sum_{u\in B(2,x)}\sum_{v\neq v'\in B(1,u)\cap B(1,x)}\frac{2}{n_u}(f(v)-f(v'))^2.
\]

We denote the $d$ elements of $B(1,x)$ as $v_1,\ldots,v_d$. 

We notice that only the $d$ values $f(v)$, $v\in B(1,x)$ appear in this last formula. We see them as $d$ variables $x_1,\ldots,x_d\in\mathbb{R}$  where $x_i=f(v_i)$ and study the minimum of this rational function.

 Both numerator and denominator are homogeneous polynomials of degree two. Therefore we can assume $\sum_{i=1}^{n}x_i^2=1$. The numerator can be seen as the polynomial associated to a quadratic form, the matrix associated to such a quadratic form is the matrix $A(x)$. This can be easily proved computing the coefficient of $x_ix_j$ in the formula for $\Ric(G)_x$ and see that it coincides with $A(x)_{ij}$.

We obtain
\[
\Ric(G)_x=\inf_{y\in\mathbb{S}^{d-1}}y^tA(x)y.
\]

From the definition of $A(x)$ we notice that it is symmetric therefore all its eigenvalues are real, in particular we can talk about the minimum eigenvalue of $A$ and call it $\lambda_{\min}(x)$. We deduce that $\Ric(G)_x=\lambda_{\min}(x)$ and conclude knowing that
$\Ric(G)=\inf_{x\in\mathcal{V}(G)}\Ric(G)_x$.

\end{proof}

\begin{example}
\label{esempio3}
Let $G$ be the Bruhat graph of $S_3=A_2$
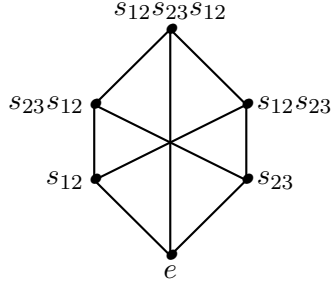
\begin{center}\label{figa}
  \begin{tikzpicture}[scale=1]
     \draw (0,0) node[anchor=north]  {$e$};
     \draw (-1,1) node[anchor=east]  {$s_{12}$};
     \draw (1,1) node[anchor=west]  {$s_{23}$};
     \draw (-1,2) node[anchor=east]  {$s_{23}s_{12}$};
     \draw (1,2) node[anchor=west]  {$s_{12}s_{23}$};
     \draw (0,3) node[anchor=south]  {$s_{12}s_{23}s_{12}$};
\foreach \Point in {(0,0), (1,1),(-1,1),(-1,2), (1,2),(0,3)}{
    \node at \Point {\textbullet};
}
    \draw[thick] (0,0) -- +(-1,1);
    \draw[thick] (0,0) -- +(1,1);
    \draw[thick] (0,0) -- +(0,3);
    \draw[thick] (1,1) -- +(-2,1);
    \draw[thick] (1,1) -- +(0,1);
    \draw[thick] (-1,1) -- +(2,1);
    \draw[thick] (-1,1) -- +(0,1);
    \draw[thick] (-1,2) -- +(1,1);
    \draw[thick] (1,2) -- +(-1,1);
  \end{tikzpicture}
\end{center}
 We fix $x=e$, in this case we have that:
\[
B(1,x)=\{s_{12},s_{23},s_{12}s_{23}s_{12}\}; \quad
B(2,x)=\{s_{23}s_{12},s_{23}s_{12}\} \quad d(x)=3.
\]
 Notice that all the vertices $x$ in $G$ have isomorphic length-2 path subgraphs, therefore $A(x)=A(x')$ for all $x,x'\in\mathcal{V}(B(W))=W$. We take $x=e$, we have 
$$
n_{u}=2 \text{ for any }u\in B(2,x)
$$
and
$$
\mathcal{U}_v=B(2,x)\text{ for any }v\in B(1,x).
$$
The graph is triangle free so $t_v$ and $T_{v,v'}$ are zero. The resulting matrix is
\[
A(e)=
\begin{bmatrix}
\frac{8}{3}&-\frac{1}{3}&-\frac{1}{3}\\
-\frac{1}{3}&\frac{8}{3}&-\frac{1}{3}\\
-\frac{1}{3}&-\frac{1}{3}&\frac{8}{3}
\end{bmatrix}
\]
whose eigenvalues are $2$ and $\frac{10}{3}$ (the latter with multiplicity two). We obtain that $\Ric(G)=2$.
\end{example}

Theorem \ref{teomatrix} is a helpful tool to compute the Ricci curvature of a finite graph. We now show how combining this result with some theorems on the eigenvalues of a matrix  we can obtain inequalities for the curvature of infinite families of graphs.

\begin{teo}\label{corostima}
Let $G$ be a triangle free graph then 
\[
\Ric(G)\geq 4-\max_{x,y\in \mathcal{E}(G)} \left( \frac{3d(x)+d(y)}{2} \right)
\]

\end{teo}
\begin{proof}
Let $x\in \mathcal{V}(G)$ and $d=d(x)$, we write $B(1,x)=\{y_1,\ldots,y_{d(x)}\}$.
Let $y_i=y$ be the element in $B(1,x)$ with maximum degree. We want to compute the Gershgorin circle (see Theorem \ref{gersh}) associated to the i-th column of $A(x)$. Its center is
\[
A_{i,i}(x)=\sum_{u\in B(2,x)\cap B(1,y)}\frac{2(n_u-1)}{n_u}+1+\frac{4-d(x)-d(y)}{2},
\]
and its radius is:
\[
\sum_{j=1,j\neq i}^{d(x)}|A_{i,j}(x)|=\sum_{j=1,j\neq i}^{d(x)}|-\sum_{u\in B(1,y)\cap B(1,y_j)\cap B(2,x)}\frac{2}{n_u}+1|.
\]
The minimal real value contained in such a circle is
\begin{align*}
A_{1,1}-\sum_{i=2}^{d(x)}|A_{i,1}|\geq &A_{1,1}-\sum_{i=2}^{d(x)}|-(\sum_{u\in B(1,y)\cap B(1,y_i)\cap B(2,x)}\frac{2}{n_u}|+1)=\\
=& \sum_{{u\in B(2,x)\cap B(1,y)}}\frac{2(n_u-1)}{n_u}+1+\frac{4-d(x)-d(y)}{2}-\\&-\sum_{{u\in B(2,x)\cap B(1,y)}}\frac{2(n_u-1)}{n_u} -d(x)+1=\\
=& \frac{4-d(x)-d(y)+2-2d(x)+2}{2}\\
=&4-\frac{3d(x)-d(y)}{2}.
\end{align*}
We conclude that the minimal eigenvalue of $A(x)$ is greater than $4-\frac{3d(x)+d(y)}{2}$. From Theorem \ref{teomatrix} we know that 
\[
\Ric(G)=\text{min}\{\lambda | \lambda \text{ is eigenvalue of some }A(x)\}\geq \min_{x\in \mathcal{V}(G),y\in B(1,x)}\{4-\frac{3d(x)+d(y)}{2}\}.
\]
The claim follows.
\end{proof}

We go on by considering a corollary of Theorem \ref{corostima} that holds for graphs with no triangles and no quadrilaterals. These are graphs such that given a path of length 
$4$ $v_1-v_2-v_3-v_4-v_5$ the five vertices are distinct. With this condition any length-2 path subgraph has no cycles.

\begin{corollary}\label{coro1}
Let $G$ be a graph with no triangles nor quadrilaterals. The following inequality then holds for $\ric(G)$:
\[
4-max_{\{x,v\}\in \mathcal{E}(G)}\left(\frac{3d(x)+d(v)}{2}\right)\leq \ric(G)\leq min(2,\frac{2+d(x')-d(v')}{2})
\]
where $x'$, $v'$ are adjacent vertices that minimizes $d(x')-d(v')$.
\end{corollary}  
\begin{proof}
We fix a vertex $x$ and $u\in B(2,x)$,
 we notice that $n_u=1$. For the same reason $\mathcal{U}_{v}\cap \mathcal{U}_{v'}=\emptyset$ for every $(v,v')$ pair of distinct elements in $B(1,x)$. 
Suppose that $|B(1,x)|=n$, we assume that these $n$ elements are ordered $B(1,x)=\{v_i\}_{i=1,\ldots,n}$.
Therefore the matrix $A(x)$ described in Theorem \ref{teomatrix} is:
\[
A(x)=
\begin{bmatrix}
d_1+1      & 1      &  \ldots & 1\\
1      & \ddots &\ddots   & \vdots \\
\vdots & \ddots & \ddots  & 1\\
1      &\ldots  & 1       & d_n+1
\end{bmatrix},
\]
where  $d_i=\frac{4-d(x)-d(v_i)}{2}$. We know that
$\ric(G)_x$ is the minimum eigenvalue of $A(x)$. We apply Gershgorin's Theorem  to this matrix to bound this value. First we consider the Gershgorin circles, these are the sets
\[
K_i=\{z\in\mathbb{C}: |z-d_i-1|\leq d-1\}.
\]
We deduce that $\ric(G)_x\geq \frac{4-d(x)-d(v)}{2}+1-d+1=\frac{8-3d(x)-d(v)}{2}$ where $v$ is the element in $B(1,x)$ with maximum degree. Also, we obtain that $\ric(G)_x\leq \frac{4-d(x)-d(v)}{2}+d(x)-1=\frac{2+d(x)-d(v)}{2}$ where $v$ is the element in $B(1,x)$ with minimum degree. Notice that because $G$ is triangle free we know from \ref{triang} that $\ric(G)\leq 2$. We obtain the following bounding for $\ric(G)$:
\[
4+\frac{-3d(x)-d(v)}{2} \leq \ric(G)\leq \min(2,\frac{2+d(x')-d(v')}{2})
\]
where $(x,v)$ is a pair of adjacent vertices that maximizes $3d(x)-d(v)$ and $x'$, $v'$ are adjacent vertices that minimizes $d(x')-d(v')$.
\end{proof}

This result can be used to bound the Ricci curvature of any tree . A new corollary follows, this one gives the exact Ricci curvature of a subfamily of graphs with no triangles nor quadrilaterals.

\begin{corollary}\label{coro2}
Let $G$ be a graph with no triangles and quadrilaterals and with the property that all the length-2 path subgraphs are isomorphic. Then
\[
\ric(G)={2-d}
\]
where $d$ is the degree of any vertex in $G$.
\end{corollary}
\begin{proof}
We fix any vertex $x$ in $G$, from the hypothesis on the length-2 path subgraphs we can say that $\ric(G)_x=\ric(G)$. The matrix $A(x)$ associated to this vertex is:
\[
A(x)=
\begin{bmatrix}
3-d      & 1      &  \ldots & 1\\
1      & \ddots &\ddots   & \vdots \\
\vdots & \ddots & \ddots  & 1\\
1      &\ldots  & 1       & 3-d
\end{bmatrix},
\]

We conclude that $A(x)$ is a circulant matrix and $\ric(G)$ is its smallest eigenvalue. By Theorem \ref{circ} the eigenvalues of $A(x)$ are:
\[
\lambda_k=3-d+\sum_{i=1}^{d-1}{\zeta}^ki=2-d\quad \zeta\text{ any primitive d-th root of unity for }k=1,\ldots,d-1
\]
\[
\lambda_0=3-d+d-1=2.
\]
By Theorem \ref{circ} we conclude that $\ric(G)=2-d$.
\end{proof}

This Corollary implies the following:

\begin{corollary}
Any integer smaller than $2$ is the discrete Ricci curvature of a graph.
\end{corollary}

 In particular given $z\in \mathbb{Z}$, $z<2$ it is the curvature of the Cayley graph of the free group generated by $2-z$ elements. 

\section{Ricci curvature of Weak orders}\label{secweak}
In this section we compute the Ricci curvature of the Hasse graphs of the weak Bruhat order associated to Coxeter groups, according to our notation we write $V(W)$ referring to the weak order graph of $W$. 
The study of the curvature of these weak order graphs, together with Theorem \ref{sgiso} allows us to obtain an isoperimetric inequality.
Furthermore the weak order graph associated to affine Weyl groups is locally finite as well, this leads to two theorems: one for finite Coxeter groups (\ref{V(W)1}) and one for affine Weyl groups (\ref{V(W)2}). The main idea behind the proof of the result in this Section is similar to the one used to prove Theorem \ref{teomatrix}: we find the minimum of a rational function studying the minimal eigenvalue of a matrix.

We are ready for the first theorem of this section, that only deals with finite irreducible Coxeter groups. A way to study the Ricci curvature of $V(W)$ if $W$ is a finite non-irreducible Coxeter group is presented in 
Theorem \ref{V(W)red}.
\begin{teo}
\label{V(W)1}
Given $(W,S)$ a finite irreducible Coxeter system and $V(W)$ the weak order graph associated to it, the following holds:
\begin{itemize}
\item $\Ric(V(W))=-2\cos(\frac{\pi}{|S|})$ if $W$ is a strictly linear Coxeter group; 
\item $-4 \leq \Ric(V(W))\leq -2$ if $W=D_n$ with $n\geq 4$ and $\Ric(V(D_3))\simeq -1.7$;
\item $\Ric(V(E_6))\simeq -2.30$, $\Ric(V(E_7))\simeq -2.33$  and $\Ric(V(E_8))\simeq -2.34$.
\end{itemize}
\end{teo}

\begin{proof}

Given an element $g\in W$ we can consider the automorphism induced on $V(W)$ by the multiplication by $g$. By Lemma \ref{stesso}
 the local Ricci curvature at a given vertex (say $e$) coincides with the global Ricci curvature of $V(W)$. We compute the former, therefore we consider the length-2 path subgraph with center $e$. First we want to study how many paths connect a given element in $u\in B(2,e)$ to $e$. Because every edge in $V(W)$ is labeled with a simple rlefection, the study of the paths correspond to the study of the equalities $ss'=zz'$ with $s,s',z,z'\in S$. 
Note that if $s_1s_2=s_3s_4\neq e$ in $W$ with $s_1,s_2,s_3,s_4\in S$ and $(s_1,s_2)\neq(s_3,s_4)$ then
$s_1=s_4$, $s_2=s_3$ and $m_{1,2}=2$.
Let $W_I$ and $W_J$ the parabolic subgroups of $W$ generated by $I=\{s_1,s_2\}$ and $J=\{s_3,s_4\}$. We notice that $W_I\cap W_{J}\neq \{0\}$, from $W_I\cap W_J=W_{I\cap J}$ it follows from \cite[Proposition 2.4.1]{BB} that $I\cap J\neq \emptyset$, this implies that either $s_1=s_4$, either $s_2=s_3$. From this it follows that $s_1=s_4$ and $s_3=s_2$.

A vertex in $B(2,e)$ is therefore connected to $e$ through one or two paths. Here follows the picture of the length-2 path subgraphs of $e$ in the weak order associated to $S_4$.
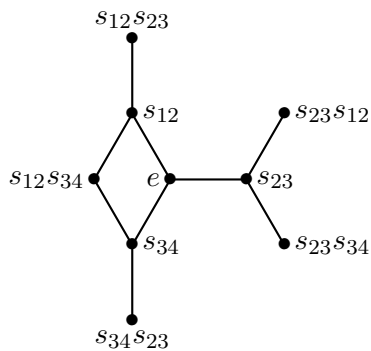
\begin{center}\label{fig}
  \begin{tikzpicture}[scale=1]
     \draw (0,0) node[anchor=east]  {$e$};
     \draw (-1,0) node[anchor=east]  {$s_{12}s_{34}$};
     \draw (1,0) node[anchor=west]  {$s_{23}$};
     \draw (-0.5,0.87) node[anchor=west]  {$s_{12}$};
     \draw (-0.5,-0.87) node[anchor=west]  {$s_{34}$};
     \draw (1.5,0.87) node[anchor=west]  {$s_{23}s_{12}$};
     \draw (1.5,-0.87) node[anchor=west]  {$s_{23}s_{34}$};
     \draw (-0.5,-1.87) node[anchor=north]  {$s_{34}s_{23}$};    
     \draw (-0.5,1.87) node[anchor=south]  {$s_{12}s_{23}$};
\foreach \Point in {(0,0), (1,0),(-1,0),(-0.5,0.87),(1.5,0.87), (-0.5,-0.87),(1.5,-0.87),(-0.5,-1.87),(-0.5,1.87)}{
    \node at \Point {\textbullet};
}
    \draw[thick] (0,0) -- +(-0.5,0.87);
    \draw[thick] (0,0) -- +(-0.5,-0.87);
    \draw[thick] (0,0) -- +(1,0);
    \draw[thick] (-1,0) -- +(0.5,-0.87);
    \draw[thick] (-1,0) -- +(0.5,0.87);
    \draw[thick] (1,0) -- +(0.5,0.87);
    \draw[thick] (1,0) -- +(0.5,-0.87);
    \draw[thick] (-0.5,-1.87) -- +(0,1);
    \draw[thick] (-0.5,1.87) -- +(0,-1);        
  \end{tikzpicture}
\end{center}

 If there is one path from $u$ to $e$, say  $u-v-e$, then $u=vv'$ and $m_{v,v'}\geq 3$; if there are two, say $u-v-e$ and $u-v'-e$, then $u=vv'$ and $m_{v,v'}=2$.

Let $u\in B(2,e)$, consider a pair of reflections in $S$, say $s$ and $s'$, we then have the following inequalities:

\begin{align*}
(f(u)-2f(s))^2+(f(u)-2f(s'))^2=& 2f(u)^2
+4f(s)^2+4f(s')^2-4f(u)f(s)-\\&4f(u)f(s')+4f(s)f(s')-4f(s)f(s')\\
=&2(f(u)-f(s)-f(s'))^2+2(f(s)-f(s'))\\
\geq &2(f(s)-f(s'))^2 
\end{align*}
If $u=ss'=s's$  and
\[
(f(u)-2f(s))^2
\geq 0 \quad \text{ if }u=ss', \text{ with }ss'\neq s's.
\]
Therefore the minimal values for $\Gamma_2(f)(e)$ are reached for $f(u)=f(s)+f(s')$ if $u=ss'=s's$ and $u=2f(s)$ otherwise. We substitute these minimal values in the formula presented in Proposition \ref{proof2} for the discrete curvature obtaining:
\begin{align}
\Ric(V(W)):=&\inf_{f} \frac{\frac{1}{2}\sum_{s\in S, u\in B(1,s)\cap B(2,e)}(f(u)-2f(s))^2+(\sum_{s\in S}f(s))^2+(2-d(e))\sum_{s\in S}f(s)^2}{\sum_{s}f(s)^2} \nonumber\\
=& \inf_{f} \frac {\sum_{\{s,s'\}\in S}(f(s)-f(s'))^2+(\sum_{s\in S}f(s))^2+(2-d(e))\sum_{s\in S}f(s)^2}{\sum_{s\in S}f(s)^2} \nonumber\\
& -\frac{\sum_{\{s,s'\}\subset S,m_{s,s'}\geq 3}(f(s)-f(s'))^2}{\sum_{s\in S}f(s)^2} \nonumber\\
=& 2- \sup_{f} \frac {\sum_{\{s,s'\}\subset S,m_{s,s'}\geq 3}(f(s)-f(s'))^2}{\sum_{s\in S}f(s)^2}. \label{fracteq}
\end{align}

We focus now on the rational function that appears in equation (\ref{fracteq}). Due to the condition $\sum_{s\in S}f(s)^2\neq 0$ we can simply study  
\[
   P_{W}(\{x_s\}_{s\in S}):=\sum_{\{s,s'\}\subset S,m_{s,s'}>2}(x_s-x_{s'})^2\quad\text{with }\sum_{s\in S}x_s^2=1.
\]

The polynomial $P_{W}(x)$ can be seen as the coordinate representation of the bilinear form induced by the matrix $M_{W}$ with
\begin{equation}\label{matrixV(W)}
[M_{W}]_{i,j}=
\begin{cases}
-1 & \text{ if }s_is_j\neq s_js_i \\
0&\text{ if }s_is_j=s_js_i\text{ and }i\neq j\\
\#\{k|s_is_k\neq s_ks_i\}& \text{ if }i=j.
\end{cases}
\end{equation}
The matrix $M_{W}$ is symmetric, therefore diagonalizable and so has a maximum eigenvalue $\lambda_{W}\in\mathbb{R}$. We notice that the structure of $M_W$ does not depend on the
order $m_{i,j}$ of  $s_is_j$ but rather on whether $m_{i,j}$ is different or not from $2$. This information is encoded by the Coxeter graph of $W$, ignoring the labels of the edges. In particular we notice that the Matrix $M_W$ has the same zero entries of the Cartan matrix associated to $W$ . We split the proof in three cases as follows:
\begin{itemize}
    \item[case 1] $W$ is a strictly linear finite Coxeter group;
    \item[case 2] $W$ is of type $D_n$, $n\geq 3$;
    \item[case 3] $W$ is of type $E_6, E_7$ or $E_8$.
\end{itemize}

\textbf{Case 1}. We can see $P$ as a coordinate representation of a quadratic form in $\mathbb{R}^{n}$ induced by the following matrix:
\[
M_{V(W)}=
\begin{bmatrix}
    1     &-1     &0      &\ldots & 0\\
    -1    &2      &\ddots &\ddots &\vdots\\
    0     &\ddots &\ddots &\ddots &0\\
    \vdots&\ddots &\ddots &2      &-1\\
     0    &\ldots &0      &-1     &1     
\end{bmatrix}
\]
By Theorem \ref{teowen} applied with
$$\alpha=\beta=1,\quad b=2, \quad a=c=-1 $$  
we conclude that the eigenvalues of $M_{V(W)}$ are
\[
 \lambda_i=2(1+\cos(\frac{i\pi}{n})) \text{ for } i=1he\ldots,n.
\] 
The maximum eigenvalue is therefore $2(1+\cos(\frac{\pi}{n}))$. So
\[
\text{Ric}(V(W))=2-2(1+\cos(\frac{\pi}{n}))=-2\cos(\frac{\pi}{n}).
\]
where $n=|S|$.

\textbf{Case 2}.  With the same procedure we find the matrix associated to the groups of type $D_n$, this is:\\
\begin{tabular}{ll}
\begin{minipage}{1\textwidth}
\[
M_{D_n}= 
\begin{bmatrix}
1&0&-1&&&&&\\
0&1&-1&&&&&\\
-1&-1&3&-1&&&&\\
\hline
&&-1&2&-1&&&\\
&&&\ddots&\ddots&\ddots&\\
&&&&-1&2&-1\\
\hline
&&&&&-1&1
\end{bmatrix}
\]
\end{minipage}
\end{tabular}
No formulas seem to be known for the eigenvalues of these matrices, but we can bound the greatest one with the help of Gershgorin's Theorem ( Theorem \ref{gersh}). The Gershgorin circles are:
\[
K_i = \begin{cases} \{z\in\mathbb{C}: |z-1|\leq 1\} &\mbox{for } i=1,2,n; \\ 
 \{z\in\mathbb{C}: |z-3|\leq 3\} & \mbox{for } i=3; \\
 \{z\in\mathbb{C}: |z-2|\leq 2\}&\mbox{ otherwise.}
 \end{cases} 
\]

We notice that $\cup_{i=1}^{n}K_i=K_4$, so $\lambda_{\max}\leq 6$. For $n=3$ and $n=4$ it is easy to compute the minimum eigenvalue of $M_{D_n}$, it is $3,7....$ and $4$ respectively.
For $n\geq 5$ we can take $e_3+e_n=[0,0,1,0,\ldots,1]^{T}$ and see that $(e_3+e_n)^{T}M_{D_n}(e_3+e_n)=4$, so $4\leq\lambda_{\max}\leq 6$. Therefore
$-4\leq$Ric$(W_{D_n})\leq-2$ for $n\geq 4$ and $\Ric(W_{D_3})=-1.7....$ for $n=3$ .

\textbf{Case 3.} For these exceptional groups it is sufficient to study the maximum eigenvalue of the following  matrices:
 \[
M_{E_6}=
\begin{bmatrix}
1&-1&0&0&0&0\\
-1&2&0&-1&0&0\\
0&0&1&-1&0&0\\
0&-1&-1&3&-1&0\\
0&0&0&-1&2&-1\\
0&0&0&0&-1&1
\end{bmatrix}
%\quad \lambda_{\max}=4,3082775...
\]
\[
M_{E_7}=
\begin{bmatrix}
1&-1&0&0&0&0&0\\
-1&2&0&-1&0&0&0\\
0&0&1&-1&0&0&0\\
0&-1&-1&3&-1&0&0\\
0&0&0&-1&2&-1&0\\
0&0&0&0&-1&2&-1\\
0&0&0&0&0&-1&1
\end{bmatrix}
%\quad \lambda_{\max}=4,33420053...
\]
\[
M_{E_8}=
\begin{bmatrix}
1&-1&0&0&0&0&0&0\\
-1&2&0&-1&0&0&0&0\\
0&0&1&-1&0&0&0&0\\
0&-1&-1&3&-1&0&0&0\\
0&0&0&-1&2&-1&0&0\\
0&0&0&0&-1&2&-1&0\\
0&0&0&0&0&-1&2&-1\\
0&0&0&0&0&0&-1&1
\end{bmatrix}
%\quad \lambda_{\max}=4,34292308...
\]
According to Sage Math  (see \cite{sage}) the maximum eigenvalues are of these matrices are respectively $\lambda_{E_6}=4,3082775...$, $\lambda_{E_7}=4,33420053...$ and $\lambda_{E_8}=4,34292308...$ and this proves the claim. 
\end{proof}

From the Theorem of classification of finite Coxeter groups we know that every finite Coxeter group is a direct product of finite irreducible Coxeter groups. To study the discrete curvature of any graph $V(W)$, W finite Coxeter group, is therefore sufficient to show how the curvature of these graphs interacts with the direct product. This is described in the following theorem.
\begin{teo}\label{V(W)red}
Let $(W,S)$ be any finite Coxeter system with $W=W_1\times\ldots\times W_k$ and $S=S_1\cup\ldots\cup S_k$ where $(W_i,S_i)$ is an irreducible Coxeter group for all $i=1,\ldots,k$.
Let $V(W)$ be the weak order graph associated to $(W,S)$, then:
$$
\Ric(V(W))=min_{1\leq i\leq k} \Ric(V(W_i)).
$$
\end{teo}

\begin{proof}
 We know that two elements $s,s'\in S$ may not commute only if they are contained in the same $S_i$. Therefore
\begin{align*}
\text{Ric}(G_W)&=2-\sup_{f}\frac{\sum_{vv'\neq v'v}(f(v)-f(v'))^2}{\sum_{v\in B(1,e)}f(v)^2}\\
&=2-\sup_{f}\frac{\sum_{i=1}^{k}\sum_{vv'\neq v'v,vv'\in S_i}(f(v)-f(v'))^2}{\sum_{v\in B(1,e)}f(v)^2}
,
\end{align*}
where $f$ the sup is over real function on $W$ such that $f(e)=0$  and $\Gamma(f)(x)\neq 0$. We can again reduce to the study of the eigenvalues of a matrix $M_W$, this is a block matrix of the following kind:

\[
M_W=
  \renewcommand{\arraystretch}{1.2}
  \left[
  \begin{array}{ c c | c c | c c }
    \multicolumn{1}{|c}{} & &  & \mc{ } &  &  \\
    \multicolumn{2}{|c|}{\raisebox{.6\normalbaselineskip}[0pt][0pt]{$M_{W_1}$}} &  & \mc{ } & \textrm{\Large{0}} &  \\
    \cline{1-2}
     & \mc{} & \ddots &    \mc{ }&  &  \\
     & \mc{} &        &\mc{\ddots} &  &  \\
    \cline{5-6}
     & \mc{\textrm{\Large{0}}} &  &  & & \multicolumn{1}{c|}{} \\
     & \mc{} &  &  & \multicolumn{2}{c|}{\raisebox{.6\normalbaselineskip}[0pt][0pt]{$M_{W_k}$}}
  \end{array}
  \right]
\]
So we have that the maximum eigenvalue of $M_W$ is $\lambda_{W}:=max_{1\leq i\leq k}\lambda_{W_i}$.
\end{proof}

We have therefore information on the Ricci curvature of the weak order graph of any finite Coxeter group. We can use this to obtain an isoperimetric inequality for these graphs:

\begin{corollary}\label{weakcoro}
Let $(W,S)$ be a finite Coxeter system with $A\subset W$ a subset of this. Then the following isoperimetric inequalities hold for $V(W)$:
\begin{itemize}
    \item If no dihedral groups appear in the irreducible decomposition of $W$: 
\[
|\partial A|\geq \frac{1}{2}\frac{|W|}{|S|^{|T|}|T|\sqrt{2|\Ric(V(W))|}}|A|\left(1-\frac{|A|}{|W|}\right);
\]
\item If at least one dihedral group appears in the irreducible decomposition of $W$:
\[
|\partial{A}|\geq \frac{1}{2}\sqrt{\frac{|W|}{|S|^{|T|}|T|}}|A|(1-\frac{|A|}{|W|}).
\]
\end{itemize}
where $T$ denotes the set of reflections.
\end{corollary}
\begin{proof}
First we notice that a weak order graph is a Cayley graph of a Coxeter group taking all the simple reflections as set of generators. From Corollary \ref{caysg} we have that $\lambda_{V(W)}\geq \frac{|W|}{|S|^{|T|}|T|}$.
We want to use this result to apply Theorem \ref{sgiso}. We therefore have to study
\[
\min\{\sqrt{\lambda},\frac{\lambda}{\sqrt{2}|\Ric(V(W))|}\}.
\]
From Theorem \ref{V(W)1} and Theorem \ref{V(W)red} we know that $\Ric(V(W))=0$ if a dihedral group appears in the irreducible decomposition of $W$, from this the second part of the statement follows. 
Assume now that the irreducible decomposition of $W$ is dihedral-free. Given an element $w\in W$, we know that it can be expressed through one of its minimal expressions as product of simple reflections. Such a minimal expression has length at most $|T|$. It follows that
\[
|W|\leq |S|+|S|^2+\ldots+|S|^{|T|}\leq |S|^{|T|}|T|.
\]
It is easy to see that $\frac{|W|}{|S|^{|T|+1}}<1$. From Theorem \ref{V(W)1} we notice that $\sqrt{2}|\Ric(V(W))|>1$, it follows that
\[
\sqrt{\frac{|W|}{|S|^{|T|+1}}}>\frac{|W|}{|S|^{|T|+1}}>\frac{|W|}{|S|^{|T|+1}\sqrt{2}|\Ric(V(W))|}.
\]
The first part of the statement follows.
\end{proof}

In the following corollary we present the same result obtained in Corollary \ref{weakcoro} in detail for finite irreducible Coxeter groups.

\begin{corollary}
The following isoperimetric inequalities hold for the weak order graph of finite irreducible Coxeter systems:
\begin{itemize}
    \item[$A_n$] 
    \[
|\partial A|\geq \frac{1}{2}\frac{(n+1)!}{(n)^{\binom{n+1}{2}}\binom{n+1}{2}2\sqrt{\cos(\frac{\pi}{n})}}|A|\left(1-\frac{|A|}{(n+1)!}\right);
\]
    \item[$B_n$]
      \[
|\partial A|\geq \frac{1}{2}\frac{2^nn!}{n^{n^2}n^22\sqrt{\cos(\frac{\pi}{n})}}|A|\left(1-\frac{|A|}{2^nn!}\right)\; n\geq2;
\]  
\item[$D_n$]
\[
|\partial A|\geq \frac{1}{2}\frac{2^{n-1}n!}{n^{n(n-1)}n(n-1)2\sqrt{2}}|A|\left(1-\frac{|A|}{2^{n-1}n!}\right)\; n\geq 3;
\]
\item[$I_2(m)$]
\[
|\partial A|\geq \frac{1}{2}\sqrt{\frac{1}{2^{m-1}}}|A|\left(1-\frac{|A|}{2m}\right);
\]
\item[$E_6$]
\[
|\partial A|\geq \frac{1}{2}\frac{518640}{12^{36}36\sqrt{4.8}}|A|\left(1-\frac{|A|}{518640}\right);
\]
\item[$E_7$]
\[
|\partial A|\geq \frac{1}{2}\frac{2903040}{12^{36}36\sqrt{4.8}}|A|\left(1-\frac{|A|}{2903040}\right);
\]
\item[$E_8$]
\[
|\partial A|\geq \frac{1}{2}\frac{696729600}{30^{120}20\sqrt{4.8}}\left(1-\frac{|A|}{696729600}\right);
\]
\item[$F_4$]
\[
|\partial A|\geq \frac{1}{2}\frac{1152}{12^{24}24\sqrt{2\sqrt{2}}}|A|\left(1-\frac{|A|}{1152}\right);
\]
\item[$H_3$]
\[
|\partial A|\geq \frac{4}{10^{15}\sqrt{2\sqrt{3}}}|A|\left(1-\frac{|A|}{120}\right);
\]
\item[$H_4$]
\[
|\partial A|\geq \frac{1}{2}\frac{14400}{30^{60}60\sqrt{2\sqrt{2}}}|A|\left(1-\frac{|A|}{14400}\right).
\]
\end{itemize}
\end{corollary}

For the last result of this Section we apply the same methods used to prove  Theorems $\ref{V(W)1}$ and $\ref{V(W)red}$ to compute the Ricci curvature of weak orders of affine Weyl groups.
 As already noticed the generators of an affine Weyl group has a finite number of generators, therefore the Hasse diagram associated to the weak order is a locally finite graph. 
\begin{teo}\label{V(W)2}
Let $(W,S)$ be an affine Weyl group, $V(W)$ the Hasse graph of the weak order associated to it. Then we have
\begin{itemize}
\item $\Ric(V(\tilde{A}_n))=2\cos(\frac{2\pi}{n}\lceil{\frac{n}{2}}\rceil)$;
\item $\Ric(V(W))=-2\cos(\frac{\pi}{|S|})$, if $W=\tilde{C}_n,\tilde{F_4},\tilde{G_2}, \tilde{A_1}$;
\item $-4\leq \Ric(V(W))\leq -2$ if $W=\tilde{B}_n$;
\item $-4\leq \Ric(V(W))\leq -2$ if $W=\tilde{D}_n$ if $n\geq 5$ and $\Ric(V(\tilde{D}_4))=-3$;
\item $\Ric(V(\tilde{E_6}))\sim -2.414 $ $,\Ric(V(\tilde{E_7}))\sim-2.36$,  $\Ric(V(\tilde{E_8}))\sim-2.34$.

\end{itemize}
\end{teo}

\begin{proof}
The first part of the proof of Theorem $\ref{V(W)1}$
holds for affine Weyl groups as well. We study therefore the eigenvalues of the matrix $M_W$ described  in \ref{matrixV(W)} depending on the group $W$. We divide the rest of the proof in five cases:
\begin{itemize}
    \item[1)] $W=\tilde{A}_n$; 
    \item[2)] $W$ strictly linear affine Weyl group;
    \item[3)] $W=\tilde{B}_n$;
    \item[4)] $W=\tilde{D}_n$;
    \item[5)] $W=\tilde{E}_6,\tilde{E}_7,\tilde{E}_8$.
\end{itemize}

\textbf{Case 1} We start with the groups  of type $\tilde{A}_n$, to study the Ricci curvature of the weak order we consider its Coxeter graph which is
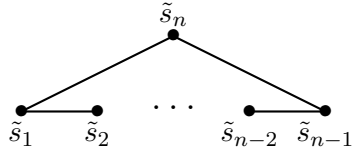
\begin{center}\label{figl}
  \begin{tikzpicture}[scale=1]
     \draw (0,0) node[anchor=north]  {$\tilde{s}_1$};
     \draw (1,0) node[anchor=north]  {$\tilde{s}_2$};
     \draw (3,0) node[anchor=north]  {$\tilde{s}_{n-2}$};
     \draw (4,0) node[anchor=north]  {$\tilde{s}_{n-1}$};
     \draw (2,1) node[anchor=south]  {$\tilde{s}_n$};
     \draw (2,-0.1) node[anchor=south]  {\textrm{\Large{$\ldots$}}};
\foreach \Point in {(0,0), (1,0),(3,0),(4,0),(2,1)}{
    \node at \Point {\textbullet};
}
    \draw[thick] (0,0) -- +(1,0);
    \draw[thick] (3,0) -- +(1,0);
    \draw[thick] (0,0) -- +(2,1);
    \draw[thick] (4,0) -- +(-2,1);
    \end{tikzpicture}
\end{center}
and deduce that $\Ric(G_{\tilde{A}_n})=2-\lambda_{\tilde{A}_n}$ where $\lambda_{\tilde{A}_n}$ is the greatest eigenvalue of the matrix:
\[
M_{\tilde{A}_n}=
\begin{bmatrix}
2&-1&&&-1\\
-1&\ddots&\ddots&\textrm{\Large{0}}&\\
&\ddots&\ddots&\ddots&\\
&\textrm{\Large{0}}&\ddots&\ddots&-1 \\
-1&&&-1&2\\
\end{bmatrix}
.
\]
This is a circulant Toeplitz matrix with $c_0=2$, $c_1=-1$, $c_{n-1}=-1$ and $c_2=\ldots=c_{n-2}=0$. From Theorem \ref{circ} we know that its eigenvalues are 
\[
\lambda_j=2-2\cos\left(\frac{2j}{n}\right)\quad 0\leq j\leq n-1. 
\]
The highest value for $\lambda_j$ is reached  for $j=\lfloor{\frac{n}{2}}\rfloor$ and $j=\lceil{\frac{n}{2}}\rceil$. 
We conclude that $\lambda_{\tilde{A}_n}=2-2\cos(\frac{2\pi}{n}\lceil{\frac{n}{2}}\rceil)=2-2\cos(\frac{2\pi}{n}\lfloor{\frac{n}{2}}\rfloor)$ and therefore $Ric(G_{\tilde{A}_n})=2\cos(\frac{2\pi}{n}\lfloor{\frac{n}{2}}\rfloor)$, in particular, it is $-2$ if $n$ is even.

\textbf{Case 2}.
$\tilde{C_n}$, $\tilde{F_4}$, $\tilde{G_2}$ and $\tilde{A_1}$ are strictly linear Coxeter groups, the curvature for their weak order graph is computed as for $A_n$ in Theorem \ref{V(W)1} and is $-2\cos(\frac{\pi}{|S|})$.

\textbf{Case 3.} The case of $\tilde{B}_n$ leads to the same unlabeled graph as the one associated to $D_n$, with $n\geq 4$. We can therefore conclude that 
$ -4\leq$Ric$(\tilde{B}_n)\leq -2$.

\textbf{Case 4.}
About the case $\tilde{D}_n$ we can see that for $n\geq 5$ the matrix associated to its quadratic form is:
\[
M_{\tilde{D}_n}=
\begin{bmatrix}
1 &0 &-1&&&&&&\\
0 &1 &-1&&&&&&\\
-1&-1&3      &\ddots&&&\textrm{\Huge{0}}&&\\
  &  & \ddots & 2 & \ddots &&&&\\
  &  &        & \ddots & \ddots & \ddots &&&\\
  &  &      &          & \ddots & 2 &  \ddots &&\\
  &  &\textrm{\Huge{0}}&          &        &  \ddots&3 &-1&-1\\
  &  &      &          &        &        &-1&1&0\\
  &  &      &          &        &        & -1&0&1  
\end{bmatrix}
\]
Its Gershgorin's circles are :
\[
K_i = \begin{cases} \{z\in\mathbb{C}| |z-1|\leq 1\} &\mbox{for } i=1,2,n-1,n \\ 
 \{z\in\mathbb{C}| |z-3|\leq 3\} & \mbox{for } i=3,n-2 \\
 \{z\in\mathbb{C}| |z-2|\leq 2\}&\mbox{ otherwise.}
 \end{cases} 
\]

By Gershgorin's Theorem we can give an upper bound to $\lambda_{\tilde{D}_n}$, namely that it is smaller than $6$. We take now $v=[0,0,1,0,\ldots,0,1]\in\mathbb{R}^{n}$, then we have that
\[
v^{T}M_{\tilde{D}_n}v=4.
\]
So, as in the case of $\tilde{B_n}$ we conclude that $-4\leq$Ric$(V(\tilde{D}_n))\leq -2$ if $n\geq 5$

We study now the case of $\Ric(W(\tilde{D_4}))$ the associated matrix is the following:

\[
M_{\tilde{D}_4}=
\begin{bmatrix}
1 &0 &-1&0&0\\
0 &1 &-1&0&0\\
-1&-1&4&-1&-1 \\
0 &0 & -1 & 1 & 0\\
0 & 0& -1 & 0 & 1
\end{bmatrix}
\]

whose eigenvalues are $[5,1,1,1,0]$. We conclude that $\Ric(V(\tilde{D}_4))=-3$.

\textbf{Case 5}.
To study $\tilde{E_6}$, $\tilde{E_7}$ and $\tilde{E_8}$
it is sufficient to study the eigenvalues of the following matrices:
\[
M_{\tilde{E_6}}=
\begin{bmatrix}
1&-1&0&0&0&0&0\\
-1&2&-1&0&0&0&0\\
0&-1&3&-1&0&-1&0\\
0&0&-1&2&-1&0&0\\
0&0&0&-1&1&0&0\\
0&0&-1&0&0&2&-1\\
0&0&0&0&0&-1&1
\end{bmatrix}
%\quad \lambda_{\max}=4,3082775...
\]
\[
M_{\tilde{E_7}}=
\begin{bmatrix}
1&-1&0&0&0&0&0&0\\
0&-1&2&-1&0&0&0&0\\
0&0&-1&2&-1&0&0&0\\
0&-0&-1&3&-1&0&0&-1\\
0&0&0&-1&2&-1&0&0\\
0&0&0&0&-1&2&-1&0\\
0&0&0&0&0&-1&1&0\\
0&0&0&-1&0&0&0&1
\end{bmatrix}
%\quad \lambda_{\max}=4,33420053...
\]
\[
M_{\tilde{E_8}}=
\begin{bmatrix}
1&-1&0&0&0&0&0&0&0\\
-1&2&-1&0&0&0&0&0&0\\
0&-1&3&-1&0&0&0&0&-1\\
0&0&-1&2&-1&0&0&0&0\\
0&0&0&-1&2&-1&0&0&0\\
0&0&0&0&-1&2&-1&0&0\\
0&0&0&0&0&-1&2&-1&0\\
0&0&0&0&0&0&-1&1&0\\
0&0&-1&0&0&0&0&1&0
\end{bmatrix}
%\quad \lambda_{\max}=4,34292308...
\]

The result follows.

\end{proof}

\section*{Acknowledgement}
The author would like to thank Francesco Brenti for suggesting her this problem and for many useful discussions. Furthermore she acknowledges the MIUR Excellence Department Project awarded to the Department of Mathematics, University of Rome Tor Vergata, CUP E83C18000100006.

\printbibliography

\end{document}